\documentclass[graybox]{svmult}

\usepackage{mathptmx}       % selects Times Roman as basic font
\usepackage{helvet}         % selects Helvetica as sans-serif font
\usepackage{courier}        % selects Courier as typewriter font
\usepackage{type1cm}        % activate if the above 3 fonts are
\usepackage{makeidx}         % allows index generation
\usepackage{graphicx}        % standard LaTeX graphics tool
\usepackage{multicol}        % used for the two-column index
\usepackage[bottom]{footmisc}% places footnotes at page bottom

\makeindex             % used for the subject index
\usepackage{amssymb,amsmath,psfrag,mathtools, comment,algorithmic}
\usepackage[ruled,vlined,linesnumbered]{algorithm2e}

\usepackage{hyperref}
\usepackage{subfigure}
\usepackage{color}

\newcommand{\negi}[1]{n_{\_}(#1)}

\DeclareMathOperator{\trace}{trace}
\newcommand{\eq}[1]{\begin{equation}\label{#1}}
\newcommand{\en}{\end{equation}}

\newcommand{\evcomment}[1]{{\color{blue}~\textsf{[EV: #1]}}}

\begin{document}

\title*{Preconditioned iterative methods for eigenvalue counts}
\author{Eugene Vecharynski and Chao Yang}
\institute{Eugene Vecharynski \at Computational Research Division, Lawrence Berkeley National Laboratory, Berkeley, CA 94720, \email{evecharynski@lbl.gov}
\and Chao Yang \at Computational Research Division, Lawrence Berkeley National Laboratory, Berkeley, CA 94720 \email{cyang@lbl.gov}}
\maketitle

%\abstract*{
%We describe preconditioned iterative methods for estimating the 
%number of eigenvalues of a Hermitian matrix within a given interval. 
%Such estimation is useful in a number of applications. It can also
%be used to develop an efficient spectrum-slicing strategy to compute
%many eigenpairs of a Hermitian matrix.  Our method is based on
%the Lanczos- and Arnoldi-type of iterations. We show that with
%a properly defined preconditioner, only a few iterations may be 
%needed to obtain a good estimate of the number of eigenvalues within
%a prescribed interval. We also demonstrate that the number of iterations 
%required by the proposed preconditioned schemes is independent of the size 
%and condition number of the matrix.  The efficiency of the methods is 
%illustrated on several problems arising from density functional theory 
%based electronic structure calculations.
%}

\abstract{
We describe preconditioned iterative methods for estimating the 
number of eigenvalues of a Hermitian matrix within a given interval. 
Such estimation is useful in a number of applications.
In particular, it can be used to develop an efficient spectrum-slicing 
strategy to compute
many eigenpairs of a Hermitian matrix.  Our method is based on
the Lanczos- and Arnoldi-type of iterations. We show that with
a properly defined preconditioner, only a few iterations may be 
needed to obtain a good estimate of the number of eigenvalues within
a prescribed interval. We also demonstrate that the number of iterations 
required by the proposed preconditioned schemes is independent of the size 
and condition number of the matrix.  The efficiency of the methods is 
illustrated on several problems arising from density functional theory 
based electronic structure calculations.
}

\section{Introduction}

The problem of estimating the number of eigenvalues of a large and sparse
Hermitian matrix $A$ within a given 
interval~$[\xi, \ \eta]$ has recently drawn a lot of attention, 
e.g.,~\cite{DiNapoli.Polizzi.Saad:2015-TR, Lin.Yang.Saad:2015-TR}.
One particular use of this estimation is in the implementation of
a ``spectrum slicing'' technique for computing many eigenpairs of a Hermitian matrix~\cite{Aktulga.Lin.Haine.Ng.Yang:2014, Li.Xi.Ve.Yang.Saad:2015-TR}. 
Approximate eigenvalue counts are used to 
determine how to divide the desired spectrum into several subintervals that
can be examined in parallel.  In large-scale data analytics, efficient
means of obtaining approximate eigenvalue counts is required for estimating the generalized rank of a given 
matrix; see, e.g.,~\cite{Zhang.Wainwright.Jordan:2015-TR}.
%which is an essential step, e.g., in the principal component analysis (PCA) or collaborative filtering 
%algorithms~\cite{?}.
   
A traditional approach for counting the number of eigenvalues of $A$ in~$[\xi, \ \eta]$ is based on the Sylevester's
law of inertia~\cite{Parlett:98}. The inertia of the shifted matrices 
$A - \xi I$ and $A - \eta I$ are obtained by performing $LDL^T$ factorizations
of these matrices~\cite{Aktulga.Lin.Haine.Ng.Yang:2014}.
This approach, however, is impractical if $A$ is extremely large or not 
given explicitly. 

Several techniques that avoid factoring $A$ have recently been described in~\cite{DiNapoli.Polizzi.Saad:2015-TR, Lin.Yang.Saad:2015-TR}.
These methods only require multiplying $A$ with a number of vectors.
In~\cite{Lin.Yang.Saad:2015-TR}, a survey that describes several approaches 
to approximating the so-called density of states (DOS), which 
measures the probability of finding eigenvalues near a given point on the real line is presented. The DOS
approximation can then be used to obtain an estimate of the number of eigenvalues in~$[\xi, \ \eta]$. The potential
drawback of a DOS estimation based approach is that, instead of directly targeting the specific interval~$[\xi, \ \eta]$, it always tries to approximate the eigenvalue distribution on the entire spectrum first.

Conceptually, the approaches in~\cite{DiNapoli.Polizzi.Saad:2015-TR, Lin.Yang.Saad:2015-TR} are based on constructing a least-squares polynomial approximation 
of a spectral filter.  Such approximations, however, often yield polynomials 
of a very high degree if $A$ is ill-conditioned or the eigenvalues to 
be filtered are tightly clustered. These are common issues in practical 
large-scale computations.  In particular, matrices originating from the
discretization of partial differential operators tend to become more 
ill-conditioned as the mesh is refined. As a result,
the polynomial methods of~\cite{DiNapoli.Polizzi.Saad:2015-TR, Lin.Yang.Saad:2015-TR} can become prohibitively expensive.  The overall cost of the computation
becomes even higher if the cost of multiplying $A$ with a vector is 
relatively high.

In this work we explore the possibility of using preconditioned iterative
methods to reduce the cost of estimating the number of eigenvalues within an 
interval.  By applying the Lanczos or Arnoldi iteration to preconditioned
matrices with properly constructed Hermitian positive definite (HPD) 
preconditioners, we can significantly reduce the number of matrix-vector 
multiplications required to obtain accurate eigenvalue counts. Furthermore,
when a good preconditioner is available, we can keep the number of 
matrix-vector multiplications (roughly) constant even as the problem size 
and conditioning of $A$ increase. 
The methods we present in this paper do not require the lower and upper 
bounds of the spectrum of $A$ to be estimated a priori. This feature compares
favorably with the methods of~\cite{DiNapoli.Polizzi.Saad:2015-TR, Lin.Yang.Saad:2015-TR} since obtaining such bounds can by itself be a challenging task. 

%The performance of the proposed schemes depends to a large extent on the 
%quality of the HPD preconditioner associated with the matrix $A - \tau I$. 
%While the development of such a preconditioner 
%is outside the scope of this paper, we point to several available options.
%%%%%%
% such as the absolute value~\cite{Ve.Kn:13} or incomplete $LDL^T$ (ILDLT) based
%preconditioning. In certain applications, such as plane wave based DFT, the SPD operators have traditionally been used as 
%preconditioners for indefinite matrices~\cite{Teter.Payne.Allan:89}. Hence they can be readily used for estimating the eigenvalue counts
%withing the presented preconditioned techniques. 
    
This paper is organized as following. Section~\ref{sec:idea} outlines the main idea, followed by derivation of the preconditioned 
Lanczos-type estimator based on Gauss quadrature in Section~\ref{sec:lan}. The preconditioned Arnoldi-type algorithm is presented 
in Section~\ref{sec:arnoldi}. In Section~\ref{sec:poly}, we discuss the proposed methods from the polynomial perspective. 
The performance of the 
%proposed 
introduced schemes depends to a large extent on the 
quality of the HPD preconditioner associated with the matrix $A - \tau I$. 
While the development of such a preconditioner 
is outside the scope of this paper, we point to several available options in Section~\ref{sec:prec}. Several
numerical experiments are reported in Section~\ref{sec:num}. 

\section{Basic idea}\label{sec:idea}

To simplify our presentation, let us assume that the endpoints $\xi$ and $\eta$ are different from any eigenvalue~of~$A$. 
Then the number of eigenvalues $c(\xi,\eta)$ of $A$ in~$[\xi, \ \eta]$
is given by the difference $c(\xi,\eta) = n_{\_}(A - \eta I) - n_{\_}(A - \xi I)$,
%\eq{eq:count}
%c(\xi,\eta) = n_{\_}(A - \eta I) - n_{\_}(A - \xi I),
%\en
where $n_{\_}(A - \tau I)$ denotes the negative inertia (i.e., the number of negative eigenvalues) of $A - \tau I$. 
Hence, in order to approximate $c(\xi,\eta)$, it is sufficient to estimate $n_{\_}(A - \tau I)$ for a given real number $\tau$.   

The problem of estimating $n_{\_}(A - \tau I)$ can be reformulated as that of approximating the 
trace of a matrix step function. Namely, let
\eq{eq:h}
h(x) = 
\left\{
\begin{array}{cl}
 1, & x < 0 \ ; \\
 0, & \mbox{otherwise} \ . 
\end{array} \right.
\en
%where $x \in [ \lambda_{\min} - \tau ,   \lambda_{\max} - \tau  ]$, with
%$\lambda_{\min}$ and $\lambda_{\max}$ being the smallest and largest eigenvalues of $A$, respectively.
Then 
%the negative inertia of $A - \tau I$ can be expressed as 
\eq{eq:negi}
n_{\_}(A - \tau I) = \trace \left\{ h(A - \tau I) \right\}. 
\en   

Now let us assume that $T$ is an HPD 
preconditioner for the shifted matrix $A - \tau I$ in the 
sense that the  spectrum of $TA$ is clustered around a few
distinct points on the real line. Specific options for constructing such preconditioners will be discussed in Section~\ref{sec:prec}.

%It is well known that $T$ can be applied from the left
%or right, yielding the left- and right-preconditioned matrices $T(A - \tau I)$ and $(A - \tau I)T$, 
%respectively; see, e.g.,~\cite{Saad:03}.
If $T$ is available in a factorized form $T= M^*M$, 
estimating $\negi{A-\tau I}$ is equivalent to estimating 
$\negi{M(A-\tau I)M^\ast}$, i.e., transforming $A-\tau I$ to
$C = M(A-\tau I)M^\ast$ preserves the inertia.  Hence, we have
\eq{eq:negi_T}
n_{\_}(A - \tau I) = \trace \left\{ h(C) \right\}.
\en 

If $T=MM^\ast$ is chosen in such a way that its spectrum has
a favorable distribution, i.e., the eigenvalues of $C$ is 
clustered in a few locations, then estimating 
$\trace \left\{ h(C) \right\}$ can be considerably easier than
estimating $\trace \left\{ h(A - \tau I) \right\}$

If the multiplication of $C$ with a vector can be performed efficiently,
then the trace of $C$ can be estimated as
\eq{eq:trace}
\trace \left\{ C \right\} \approx \frac{1}{m}\sum_{j=1}^m v_j^* C v_j,
\en 
where the entries of each vector $v_j$ are i.i.d. random variables with zero mean and unit variance; 
see~\cite{Hutchinson:89, Avron.Toledo:2011}.  It follows that
\eq{eq:traceT}
n_{\_}(A - \tau I) = \trace \left\{ h(C) \right\} \approx \frac{1}{m}\sum_{j=1}^m v_j^* h(C) v_j,
\en 
for a sufficiently large sample size $m$.

The variance of the stochastic trace estimator is known to depend on the magnitude of off-diagonal entries
of the considered matrix, which is $h(C)$ in~\eqref{eq:traceT}. 
%(\CY{$h(C)$ is not sampled} 
Clearly, different choices of the preconditioned 
operator $C$ yield different matrices $h(C)$, and hence lead to different convergence rates of
the estimator~\eqref{eq:traceT}.     

\section{Preconditioned Lanczos}\label{sec:lan}

If $A$ is large, then the exact evaluation of $h(C)$ in~\eqref{eq:traceT} 
can be prohibitively expensive, because it requires a full eigendecomposition of the preconditioned matrix. 
A more practical approach in this situation would be to (approximately) compute 
$v^* h(C) v$ for a number of randomly sampled vectors $v$ without explicitly 
evaluating the matrix function. 

\subsection{The Gauss quadrature rule}
Let us assume that $T = M^*M$ is available in the factorized form and let $C = M(A - \tau I)M^*$
in~\eqref{eq:traceT}. We also assume that the Hermitian matrix $C$ 
has $p \leq n$ distinct eigenvalues $\mu_1 < \mu_2 < \ldots < \mu_p$. 
%of multiplicities $m_1, m_2, \ldots, m_p$, respectively.
%Then each $\mu_i$ has an associated eigenspace of dimension $m_i$, whose orthonormal basis is denoted 
%by $U_i$. Thus, $U = [U_1, U_2, \ldots, U_p]$ is an $n$-by-$n$ matrix of eigenvectors from the 
%full eigendecomposition of $M(A - \tau I)M^*$.    

%Since $M(A - \tau I)M^*$ is Hermitian, its eigendecomposition is of the form 
%\[
%M(A - \tau I)M^* = W \Sigma W^*, \quad \Sigma = \mbox{diag} \left\{ \mu_1, \ldots  \mu_n \right\}, \quad
%W^*W = I,
%\] 
%%be the eigendecomposition of the Hermitian 
%%preconditioned matrix $M(A - \tau I)M^*$, 
%where $\Sigma$ is the diagonal matrix of eigenvalues $\mu_1 \leq \mu_2 \leq \ldots \leq \mu_n$
%of $M(A - \tau I)M^*$ and columns of $W$ contain the associated orthonormal eigenvectors. 
%Then 
%With the above notation, 
Consider the orthogonal expansion of $v$ in terms of the 
eigenvectors of $C$, i.e., $v = \sum_{i=1}^p \alpha_i u_i$, 
where $u_i$ is an normalized eigenvector associated with the 
eigenvalue $\mu_i$, and $\alpha_i = u_i^* v$. 
It is then easy to verify that
\eq{eq:wsum}
v^* h(C) v =  \sum_{i=1}^p \alpha_i^2 h(\mu_i)  \equiv \sum_{i=1}^{p\_} \alpha_i^2, \quad \alpha_i^2 = |u_i^* v|^2,
\en
where $p\_$ denotes the number of negative eigenvalues. 
%and $e_l$ is the $l$-th column of the identity matrix of size $m_i$.
%Furthermore, 
The right-hand side in~\eqref{eq:wsum} can be viewed as a Stieltjes integral of the step function $h$
with respect to the measure defined by the piecewise constant function 
%\CY{not sure whether this is correct, the measure should be a point measure}
\eq{eq:alpha}
\alpha_{C, v}(x) = \left\{ 
\begin{array}{ll}
0, & \mbox{if} \; x < \mu_1, \\
%\displaystyle 
\sum_{j=1}^i \alpha_j^2, & \mbox{if}  \; \mu_i \leq x < \mu_{i+1}, \\ 
%\displaystyle 
\sum_{j=1}^i \alpha_j^2, & \mbox{if}  \; \mu_p \leq x.
\end{array}
\right.
\en
Therefore, using~\eqref{eq:alpha}, we can write~\eqref{eq:wsum} as
\eq{eq:int}
v^* h(C) v =  \int h(x) d \alpha_{C, v} (x) \equiv 
\int_{\mu_1}^0 d \alpha_{C, v} (x).
\en 
%where the measure $\mu_{M(A - \tau I)M^*, v} (x)$ is a nondecreasing piecewise constant function, 
%defined on a real axis, which has jumps $\mu_j$ at the eigenvalues $\nu_j$ of the preconditioned
%matrix. 
%\evcomment{I am not sure if I will need the second equalities in the above two formulas, 
%but will keep for now}   
%Observation:The measure depends only on the first components of the eigenvectors of $M(A - \tau I)M^*$
%associated with the $p$ negative eigenvalues and the vector $v$. 

Computing the above integral directly is generally infeasible because the 
measure~\eqref{eq:alpha} is defined in terms of the unknown eigenvalues of $C$. 
Nevertheless, the right-hand side of~\eqref{eq:int} can be approximated by
using the Gauss quadrature rule~\cite{Golub.Meurant:10}, so that
\eq{eq:quad}
v^* h(C) v \approx \sum_{i=1}^k w_i h(\theta_i) \equiv \sum_{i=1}^{k\_} w_i,
\en
where the $k$ nodes $\theta_1 \leq \theta_2 \leq \ldots \leq \theta_k$ and weights $w_1, w_2, \ldots, w_k$ of the quadrature 
are determined from $k$ steps of the Lanczos procedure (see Algorithm~\ref{alg:lan}) applied to the preconditioned matrix $C$ with the starting vector $v$. 
%The $k$ steps of the Lanczos method are summarized in Algorithm~\ref{alg:lan} below. 
In~\eqref{eq:quad}, $k\_$ denotes the number of negative nodes $\theta_i$. 
\begin{algorithm}[!htbp]
\begin{smaller}
\begin{center}
  \begin{minipage}{5in}
\begin{tabular}{p{0.5in}p{4.5in}}
{\bf Input}:  &  \begin{minipage}[t]{4.0in}
                  Matrix $A - \tau I$, $T = M^*M$, starting vector $v$, and number of steps $k$.
                  \end{minipage} \\
{\bf Output}:  &  \begin{minipage}[t]{4.0in}
                  Tridiagonal matrix $J_{k+1,k}$ and the Lanczos basis $Q_{k+1} = [q_1, q_2, \ldots, q_{k+1}]$.  
                  \end{minipage}
\end{tabular}
\begin{algorithmic}[1]
\STATE $q_1 \gets v/\|v\|$; $q_0 \gets 0$; $\beta_1 \gets 0$; $Q_1 \gets q_1$; 
\FOR {$i = 1 \rightarrow k$}
     \STATE $w \gets M(A - \tau I)M^* q_i - \beta_i q_{i-1}$; 
     \STATE $\alpha_i \gets q_i^* w$; $w \gets w - \alpha_i q_i$; 
     \STATE Reorthogonalize $w \gets w - Q_i (Q_i^* w)$;
     \STATE $\beta_{i+1} \gets \|w\|$; $q_{i+1} \gets w/\beta_{i+1}$; $Q_{i+1} \gets [Q_i, \ q_{i+1}]$;  
\ENDFOR
%\STATE Return $J_{k+1,k}$ in~\eqref{eq:Tk} and $Q_{k+1}$.
\end{algorithmic}
\end{minipage}
\end{center}
\end{smaller}
  \caption{The Lanczos procedure for $M(A - \tau I)M^*$}
  \label{alg:lan}
\end{algorithm}

Specifically, given $q_1 = v/\|v\|$, running $k$ steps of the Lanczos procedure in Algorithm~\ref{alg:lan} yields the relation  
\eq{eq:lan}
%M(A - \tau I)M^* Q_k = Q_k J_k + \beta_{k+1} q_{k+1} e^*_k, \quad Q_{k+1}^*Q_{k+1} = I,
C Q_k = Q_{k+1} J_{k+1,k}, \quad Q_{k+1}^*Q_{k+1} = I,
\en
where $J_{k+1,k}$ is the tridiagonal matrix
\begin{equation}
\label{eq:Tk}
J_{k+1,k}  =  \left [ \begin{array}{cccc}
\alpha_1  & \beta_2  & &  \\
\beta_2   & \alpha_2 & \ddots &  \\
          & \ddots   & \ddots & \beta_{k}  \\
          &                        & \beta_{k}   & \alpha_k \\
          &                        &             & \beta_{k+1} \\
\end{array} \right ] \in \mathbf{R}^{(k+1) \times k}.
\end{equation}
The eigenvalues of the leading $k\times k$ submatrix of $J_{k+1,k}$, 
denoted by $J_k$, are ordered so that $\theta_1 \leq \theta_2 \leq \ldots \leq \theta_{k\_} < 0 \leq  \theta_{k\_+1} \leq \ldots \leq \theta_{k}$. 
Then the Gauss quadrature rule on the right-hand side of~\eqref{eq:quad} is defined by 
eigenvalues and eigenvectors of $J_k$, i.e.,
\eq{eq:gauss_quad}
v^* h(C) v \approx \|v\|^2 e_1^* h(J_k) e_1 = \sum_{i=1}^{k} w_i h(\theta_i) \equiv \sum_{i=1}^{k\_} w_i, \quad w_i = \|v\|^2 |z_i(1)|^2,
\en
where $z_i$ is the eigenvector of $J_k$ associated with the eigenvalue $\theta_i$, 
$z_i(1)$ denotes its first component~\cite{Golub.Meurant:10}, and
$k_{\_}$ denotes the number of negative Ritz values.
% resulting from $k$ steps
%of the Lanczos process.  

%If the spectrum of $M(A-\tau I)M^*$ is spread widely along the real line and has an arbitrary structure, which is typically the case
%when $M = I$ or the preconditoner $T = M^*M$ is poorly chosen, then the above described approach will be inefficient, as 
%the quadrature rule~\eqref{eq:gauss_quad} will likely require a large number $k$ of quadrature nodes (or, equivalently, Lanczos steps) 
%to deliver a sufficiently accurate approximation of the integral. However, a good choice of the preconditoner will ensure that 
%the spectrum of $M(A-\tau I)M^*$ is clustered in a few locations of the real axis. The Gauss quadrature effectively takes advantage
%of this~property.

If the preconditioner $T=MM^\ast$ is chosen in such a way that the spectrum of 
$C = M(A-\tau I)M^*$ is concentrated within small intervals 
$[a, b] \subset (-\infty,0)$ and $[c,d] \subset (0,\infty)$,
then, by~\eqref{eq:alpha}, the measure $\alpha_{M(A - \tau I)M^*, v}$ will have jumps 
inside $[a, b]$ and $[c, d]$, and will be constant elsewhere. Hence, the integral in~\eqref{eq:int} will be determined only by integration
over $[a, b]$ because $h$ vanishes in $[c,d]$.
Therefore, in order for quadrature rule~\eqref{eq:quad} to be a good approximation
to \eqref{eq:int}, its nodes should be chosen inside $[a, b]$. 
%This is consistent 
%with a behavior of the Lanczos procedure 
%in Algorithm~\ref{alg:lan}, which will generate Ritz values that tend to approach $[a, b]$ starting from early iterations.
%,
%%, given that
%provided that the spectrum of the preconditioned matrix is clustered in $[a, b] \cup [c, d]$.
%%, provided 
%%that the preconditioner $T = M^*M$ is sufficiently strong. 

In the extreme case in which clustered eigenvalues of $C$ coalesce into 
a few eigenvalues of higher multiplicities, the number of Lanczos steps 
required to obtain an accurate approximation in~\eqref{eq:gauss_quad} is 
expected to be very small.

%This observation is 
%formalized in the following proposition. 
\begin{proposition}\label{prop:mult}
Let the preconditioned matrix $C = M(A - \tau I)M^*$ have $p$ distinct eigenvalues. Then the Gauss quadrature~\eqref{eq:gauss_quad}
will be exact with at most $k = p$ nodes.
\end{proposition}
\begin{proof}
Let $v = \sum_{i=1}^p \alpha_i u_i$, where $u_i$ is an eigenvector of $C$ associated with the eigenvalue $\mu_i$. 
Then $p$ steps of Lanczos process with $v$ as a starting vector produce a tridiagonal matrix $J_p$ 
and an orthonormal basis $Q_p$, such that the first column of $Q_p$ is $\hat v = v/\|v\|$.
The eigenvalues $\theta_i$ of $J_p$ are exactly the $p$ distinct eigenvalues of~$C$. 
The eigenvectors $z_i$ of $J_p$ are related to those of $C$ as $u_i = Q_p z_i$. Thus, we have 
%\[
$w_i = \|v\|^2 |z_i(1)|^2 = \|v\|^2 |\hat v^* u_i|^2 = |v^* u_i|^2$,
%\]
and, by comparing with~\eqref{eq:wsum}, we see that the quadrature~\eqref{eq:gauss_quad} 
%takes the form
%\[
%\|v\| e_1^* h(J_p) e_1 = \sum_{i=1}^p   |v^* u_i|^2 h(\mu_i),
%\]
%which, by~\eqref{eq:wsum}, is
gives the exact value of $v^* h(M(A - \tau I)M^*) v$.
\end{proof}   
Proposition~\ref{prop:mult} implies that in the case of an ideal preconditioner, where $M(A - \tau I)M^*$ has 
two distinct eigenvalues, the Gauss quadrature rule~\eqref{eq:gauss_quad} is guaranteed to be exact after at most two Lanczos steps. 
%\evcomment{conclude that we expect fewer steps if good preconditoner used}

\subsection{The algorithm}
%Let $v_j$ be a vector whose entries are i.i.d. random variables from standard normal distribution. 
Let $J^{(j)}_k$ denote the $k$-by-$k$ tridiagonal matrix resulting from the $k$-step Lanczos procedure applied to 
$C=M(A-\tau I)M^*$ with a random starting vector $v_j$.
% whose entries are i.i.d. random variables from standard normal distribution. 
Assume that $k_j$ is the number of its negative eigenvalues. Then, by~\eqref{eq:traceT} and~\eqref{eq:gauss_quad}, 
the quantity $n\_(A - \tau I)$ can be approximated from the estimator 
%$L_{\tau}(k,m)$ of the form 
\eq{eq:rand_prec_quad}
%n_{\_}(A - \tau I) \approx 
L_{\tau}(k,m) = \frac{1}{m} \sum_{j=1}^m  \sum_{i=1}^{k_{j}}   w^{(j)}_i, \quad w^{(j)}_i = \|v_j\|^2|z_i^{(j)}(1)|^2, \quad v_j 
\in \mathcal{N}(0,I),
%\quad E_m \approx n_{\_}(A - \tau I),  
\en
where $z_i^{(j)}(1)$ denotes the first components of a normalized eigenvector $z_i^{(j)}$ of $J_{k}^{(j)}$
associated with the negative eigenvalues.
It is expected that, for a sufficiently large $m$,  $L_\tau(k,m) \approx n_{\_}(A - \tau I)$.
% of sampling vectors
%the tridiagonal matrix $J_{k}^{(j)}$ associated with its $k_{j}$ negative eigenvalues. Here it is 
%assumed that $J_{k}^{(j)}$ is generated by applying $k$ steps of the Lanczos process to the preconditioned matrix $M(A-\tau I)M^*$ with a 
%random starting vector $v_j$.    
%%result form the $k$-step Lanczos proces for $v_j$, as descussed above.    
%%We are now ready to state 
The expression~\eqref{eq:rand_prec_quad} is what Algorithm~\ref{alg:planczos} uses
to estimate the number of eigenvalues of $A$ that are to the left of $\tau$.
\begin{algorithm}[!htbp]
\begin{smaller}
\begin{center}
  \begin{minipage}{5in}
\begin{tabular}{p{0.5in}p{4.5in}}
{\bf Input}:  &  \begin{minipage}[t]{4.0in}
                  Matrix $A$, shift $\tau$, HPD preconditioner $T = M^*M$ for $A - \tau I$, number of steps $k$, and parameter $m$.
                  \end{minipage} \\
{\bf Output}:  &  \begin{minipage}[t]{4.0in}
                  approximate number $C_{\tau}$ of eigenvalues of $A$ that are less than $\tau$;
                  \end{minipage}
\end{tabular}
\begin{algorithmic}[1]
\STATE $L_{\tau} \gets 0$.
\FOR {$j = 1 \rightarrow m$}
   \STATE Generate $v \sim \mathcal{N}(0,I)$.
   \STATE Run $k$ steps of Lanczos process in Algorithm~\ref{alg:lan} with 
   the starting vector $v$ to obtain tridiagonal matrix $J_k$.
%   \STATE Use $J_{k+1,k}$ to construct $\tilde J_{2k-1}$ in~\eqref{eq:GA_Tk}.
   \STATE Find the eigendecomposition $(\Theta,Z)$ of $J_{k}$. 
   Let $z_1, \ldots, z_{k\_}$ be unit eigenvectors associated with negative eigenvalues. 
   \STATE Set $L_{\tau} \gets L_{\tau} + \|v\|^2 \sum_{i = 1}^{k\_} w_i$, where $w_i = |z_i(1)|^2$.
\ENDFOR
\STATE Return $L_{\tau} \gets \left[ L_{\tau}/m \right ]$.
\end{algorithmic}
\end{minipage}
\end{center}
\end{smaller}
  \caption{The preconditioned Lanczos-type estimator for $n\_(A - \tau I)$}
  \label{alg:planczos}
\end{algorithm}

In order to estimate the number of eigenvalues in a given interval $[\xi,\eta]$, Algorithm~\ref{alg:planczos} should be applied twice with 
$\tau = \xi$ and $\tau = \eta$. The difference between the estimated 
$n\_(A-\xi I)$ and $n\_(A-\eta I)$ yields the desired count.
The two runs of Algorithm~\ref{alg:planczos} generally require two different HPD preconditioners, one for $A - \xi I$ and the other for $A - \eta I$.
In some cases, however, it can be possible to come up with a single preconditioner that works well for both runs.
%, which
%can generally be different from one another.
%, although the same preconditioning operator can be as well used in both runs.  
%\CY{not sure what this means} 

The cost of Algorithm~\ref{alg:planczos} is dominated by computational work required to perform the preconditioned matrix-vector multiplication of $M(A-\tau I)M^* v$ 
at each iteration of the Lanczos procedure.
% in Algorithm~\ref{alg:lan}.  
The eigenvalue decomposition of the tridiagonal matrix 
$J_{k}$, as well as reorthogonalization of the Lanczos basis in step 6 of Algorithm~\ref{alg:lan}, 
%that is performed in practice
%to maintain orthogonality of $Q_i$, 
is negligibly small for small values of $k$, which can be ensured by a sufficiently
high quality preconditioner. Note that, in exact arithmetic, the Lanczos basis $Q_i$ should be 
orthonormal~\cite{Parlett:98}. However, in practice, the orthogonality may be lost; therefore, we reorthogonalize 
$Q_i$ at every iteration of Algorithm~\ref{alg:lan}.

%\evcomment{No  prior info is needed other than $T$, discuss stopping based on tracking stagnation of the quadrature}

%Advantages: no need to estimate the spectrum, automated stopping. Polynomial degree does not grow as the probem size inceases,
%as will be demonstrated in the numerical experiments. 

%The reorthogonalization step will not be expensive in our context as the number $k$ of Lanczos iterations will be low
%for a sufficiently good preconditioner.

\subsection{Bias of the estimator}
A relation between the Gauss quadrature~\eqref{eq:gauss_quad} and matrix functional $v^* h(C) v$ can be expressed as 
\[
\|v\|^2 \sum_{i=1}^{k\_} w_i = v^* h(C) v + \epsilon_k,
\]
where $\epsilon_k$ is the error of the quadrature rule. Thus,~\eqref{eq:rand_prec_quad} 
%of $n_{\_}(A - \tau I)$ 
can be written as
\eq{eq:bias}
L_{\tau}(k,m) = \frac{1}{m} \sum_{j=1}^{m} v_j^* h(C) v_j + \frac{1}{m} \sum_{j=1}^m \epsilon_k^{(j)},
\en
where $\epsilon_k^{(j)}$ denotes the error of the quadrature rule for 
$v_j^* h(C) v_j$. As $m$
increases, the first term in the right-hand side of~\eqref{eq:bias} converges to  
$\trace \left\{ h(C) \right\}$ = $n_{\_}(A - \tau I)$. Thus, $L_{\tau}(k,m)$ is a biased estimate of 
$n_{\_}(A - \tau I)$, where the bias is determined by the (average) error of the quadrature rule, given by the second term
in the right-hand side of~\eqref{eq:bias}. In other words, the accuracy of $L_{\tau}(k,m)$ generally depends on how well
the Gauss quadrature captures the value of the matrix functional $v^* h (M (A - \tau I) M^*) v$.  

Bounds on the quadrature error for a matrix functional $v^* f(C) v$, where $f$ is a sufficiently smooth function and $C$ is a 
Hermitian matrix, are well known. In particular, the result of~\cite{Calvetti.Golub.Reichel:99} gives the bound 
%in terms of
%a higher-order derivative of $f$ and quantitites generated by the Lanczos procedure:
\eq{eq:qbound}
|\epsilon_k| \leq \frac{N_k}{2k!} \beta_{k+1}^{2} \beta_{k}^{2} \ldots \beta_2^{2},
\en       
where the constant $N_k$ is such that $ |f^{(2k)}(x)| \leq N_k$ for $x$ in the interval containing spectrum of $C$, 
and $\beta_j$ are the off-diagonal entries of~\eqref{eq:Tk}. 

Function $h(x)$ in~\eqref{eq:h} is discontinuous. Therefore, bound~\eqref{eq:qbound} does not directly apply 
%to the quadrature error 
%for the matrix functional 
to measure the quadrature error the functional $v^* h(M(A - \tau I) M^*) v$. 
However, since the rule~\eqref{eq:gauss_quad} depends on the values of $h(x)$ only at the 
Ritz values $\theta_i$ generated by the Lanczos process for $M(A-\tau I)M^*$, it will yield exactly the same result
for any function $\tilde h(x)$, such that $\tilde h(\theta_i) = h(\theta_i)$ for all $\theta_i$. 
If, additionally, $\tilde h(x)$ assumes the same values as $h(x)$ on the spectrum of $M(A - \tau I) M^*$,
then, by~\eqref{eq:wsum}, the functionals $v^* h(M(A - \tau I) M^*) v$ and $v^* \tilde h(M(A - \tau I) M^*) v$
will also be identical. 
Hence, the quadrature errors for $v^* \tilde h(M(A - \tau I) M^*) v$ and $v^* h(M(A - \tau I) M^*) v$ will coincide. 
But then we can choose $\tilde h(x)$ as a $2k$ times continuously differentiable function
and apply~\eqref{eq:qbound} to bound the quadrature error for $v^* \tilde h(M(A - \tau I) M^*) v$. This error will be
exactly the same as that of the quadrature~\eqref{eq:gauss_quad} for $v^* h(M(A - \tau I) M^*) v$, which we are interested in.

%Thus, if we choose $\tilde h(x)$ as a sufficiently smooth function that has the derivative of order $2k$,
%then~\eqref{eq:qbound} can be readily applied to bound the quadrature error for $v^* \tilde h(M(A - \tau I) M^*) v$, which will be
%exactly the same as in quadrature~\eqref{eq:gauss_quad} for functional $v^* h(M(A - \tau I) M^*) v$.  

In particular, let us assume that the eigenvalues of $M(A-\tau I)M^*$ and Ritz values $\theta_i$ are located in 
intervals $[a,b)$ and $(c,d]$ to the left  and right of origin, respectively. 
Then we can choose $\tilde h(x)$ such that it is constant one on $[a,b)$ and constant zero on $(c,d]$. On the interval $[b,c]$,
which contains zero, we let $\tilde h(x)$ to be a polynomial $p(x)$ of degree $4k+1$, such that $p(b) = 1$, 
$p(c) = 0$, and $p^{(l)}(b) = p^{(l)}(c) = 0$ for $l = 1,\ldots, 2k$. This choice of polynomial will ensure
that the piecewise function $\tilde h(x)$ is $2k$ times continuously differentiable. 
(Note that $p(x)$ can always be be constructed by (Hermite) interpolation with the nodes $b$ and $c$; see, e.g.,~\cite{Powell:81}.)
We then apply~\eqref{eq:qbound} to obtain the bound on the quadrature error for $v^* \tilde h(M(A-\tau I)M^*) v$. 
As discussed above, this yields the estimate of the error $\epsilon_k$ of quadrature rule~\eqref{eq:gauss_quad} 
for functional $v^* h(M(A - \tau I) M^*) v$. Thus, we can conclude that the latter is bounded by~\eqref{eq:qbound}, where
%, by construction, 
$N_k$ is the maximum of $|p^{(2k)}(x)|$ on the interval $[b,c]$. 
%\CY{not sure if this argument is correct}
%\evcomment{basically saying that if spectrum is away from the origin then integration of step function is viewed as that of a very smooth, slowly
%changing function whose quadrature is fast to converge.}

This finding shows that we can expect that~\eqref{eq:gauss_quad} provides a better approximation of $v^* h(M(A - \tau I) M^*) v$ 
when the intervals $[a,b)$ and $(c,d]$, containing eigenvalues of $M(A - \tau I)M^*$ along with the Ritz values produced 
by the Lanczos procedure, are bounded away from zero. In this case, the rate of change of the polynomial $p(x)$ on $[b,c]$ will not be 
too high, resulting in a smaller value of $N_k$ in~\eqref{eq:qbound}. 

Fortunately, a good choice of the preconditioner $T = M^*M$ can
ensure that eigenvalues of $M(A - \tau I)M^*$ are clustered and away from zero. 
In this case, the Ritz values typically converge rapidly to these eigenvalues
after a few Lanczos steps. Thus, with a good preconditioner, 
the Gauss quadrature~\eqref{eq:gauss_quad} can effectively approximate the matrix functional $v^* h(M(A - \tau I) M^*) v$,
yielding small errors $\epsilon_k$ for a relatively small number of quadrature nodes. 
As a result, the 
bias of the estimator $L_{\tau}(k,m)$ in~\eqref{eq:bias} will be small and, as confirmed by numerical experiments in
Section~\ref{sec:num}. 

%\begin{proposition}\label{prop:bias_smooth} 
%Let $f(B)v_j = p_{k-1}^{(j)} (B) v_j + e_j$, where $p_{k-1}^{j}$ is the interpolating polynomial of $2k$ times continuously differentiable 
%function $f$. Then
%\[
%\frac{1}{m} \sum_{j=1}^m v_j^* f (B) v_j = \frac{1}{m} \sum_{j=1}^m v_j^* p^{j}_{k-1} (B) v_j + \tau_m,
%\]
%with
%\[
%|\tau_m| \leq \frac{C_k}{2k!}\left[ \frac{1}{m} \sum_{j=1}^m (\beta_k^{(j)} \beta_{k-1}^{(j)} \ldots \beta_1^{(j)})^2 \right],   
%\]
%where $|f^{(2k)}(x)| \leq C_k$ and $\beta_1^{(j)}, \ldots, \beta_k^{(j)}$ are the subdiagonal entries of the tridiagonal matrix
%$T_{k}$ obtained from the Lanczos process with the starting vector $v_j$.  
%\end{proposition} 
%\begin{proof}
%The quantity $v_j^* e_j$ can be viewed as an error of the Gauss quadrature. Then, using the result of~\cite{Calvetti.Golub.Reichel:99},
%\[
%|v_j^* e_j| \leq  \frac{C_k}{2k!} (\beta_k^{(j)} \beta_{k-1}^{(j)} \ldots \beta_1^{(j)})^2.
%\]
%Averaging for all $v_j$ gives the desired result. 
%\end{proof}
%The bound of the proposition does not depend on the problem size, but requires that $f$ is sufficiently smooth, which is not the 
%case for~\eqref{eq:h}.  

\subsection{The generalized averaged Gauss quadrature rule}\label{sec:ga}
The Gauss quadrature rule~\eqref{eq:gauss_quad} is exact for all polynomials of degree at most 
%$2k-1$~\cite{Golub.Meurant:10, Golub.Welsch:69}. 
$2k-1$; e.g.,~\cite{Golub.Meurant:10}. 

In the recent work of~\cite{Reichel.Spalevic.Tang:15} (and references therein), 
a so-called generalized averaged (GA) Gauss quadrature rules was introduced. 
This quadrature rule make use of the same information returned by a $k$-step Lanczos 
process, but gives an exact integral value for polynomials of degree $2k$.
Hence it is more accurate at essentially the same cost.
%This is similar to the Gauss--Kronrod rule~\cite{?}, with the difference that the GA Gauss quadrature exists
%when the latter does not.     

When applying the GA Gauss quadrature rule to the matrix functional 
$v^*h(C)v$ in~\eqref{eq:int}, we still use the expression~\eqref{eq:gauss_quad}, 
except that we have $(2k-1)$ nodes
$\theta_1, \theta_2, \ldots,$ $\theta_{2k-1}$ which are the eigenvalues of the matrix 
%and the same number of associated weights 
%are determined from the extended matrix  
%\eq{eq:ga_gauss_quad}
%v^* h(M(A - \tau I)M^*) v \approx \|v\|^2 e_1^* h(\tilde J_{2k-1}) e_1 = \sum_{i=1}^{2k-1} w_i h(\theta_i) \equiv \sum_{i=1}^{k_{\_}^{\prime}} w_i, \quad w_i = \|v\|^2 |z_i(1)|^2,
%\en
%where $\theta_1, \theta_2, \ldots, \theta_{2k-1}$ are eigenvalues of the $(2k-1)$-by-$(2k-1)$ tridiagonal matrix
\begin{equation}
\label{eq:GA_Tk}
\tilde J_{2k-1}  = \mbox{tridiag} \left\{  (\alpha_1, \ldots, \alpha_k, \alpha_{k-1}, \ldots \alpha_1), (\beta_2, \ldots, \beta_k, \beta_{k+1}, \beta_{k-1} \ldots \beta_2) \right\} 
%\left [ \begin{array}{ccccccccc}
%\alpha_1  & \beta_2  & & & & & & &\\
%\beta_2   & \alpha_2 & \ddots & & & & &\\
%          & \ddots   & \ddots & \beta_{k} & & & & & \\
%          &                        & \beta_{k}   & \alpha_k      &  \beta_{k+1}  & & & & \\
%          &                        &             & \beta_{k+1}   &  \alpha_{k-1} & \beta_{k-1} & & &\\
%          &                        &             &               &  \beta_{k-1} & \alpha_{k-2} & \beta_{k-2} & &\\
%          &                        &             &               &              & \ddots & \ddots   & \ddots & \\
%          &                        &             &               &              &  &  \beta_{3}  & \alpha_{2} & \beta_{2} \\
%          &                        &             &               &              &               &       & \beta_{2} & \alpha_{1}\\
%\end{array} \right ], 
%%\quad 1 \leq r \leq k-1,
\end{equation}
obtained from $J_{k+1,k}$ in~\eqref{eq:Tk} by extending its tridiagonal part in a ``reverse'' order. The set $(\alpha_i)$ of 
numbers in~\eqref{eq:GA_Tk} gives the diagonal entries of $J_{k+1,k}$, whereas $(\beta_i)$ define the upper and lower diagonals.   
%Similar to standard Gauss quadrature, 
Similarly, the associated weights $w_i$ are determined by squares of the first components of the properly normalized eigenvectors 
$z_i$ of $\tilde J_{2k-1}$ associated with the eigenvalues $\theta_i$;
see ~\cite{Reichel.Spalevic.Tang:15} for more details. 
% For more details, we refer the reader to~\cite{Reichel.Spalevic.Tang:15}.
%; $k^{\prime}_{\_}$ denotes the number of negative 
%eigenvalues of $\tilde J_{2k-1}$. 
%
%Note that the eigenvalues of the tridiagonal matrix $J_k$ are also the eigenvalues of $\tilde J_{2k-1}$~\cite{Reichel.Spalevic.Tang:15}. 
%Hence the $k$ nodes of the Gauss quadrature~\eqref{eq:gauss_quad} form a subset of nodes in the GA Gauss quadrature rule~\eqref{eq:ga_gauss_quad}.   
%
%The above suggests that 
%with a simple modification of Algorithm~\ref{alg:planczos} 
Thus, we can expect to increase accuracy of the estimator by a minor modification of Algorithm~\ref{alg:planczos}.
This modification will only affect step 5 of the algorithm, where $J_k$ must be replaced by the extended tridiagonal matrix~\eqref{eq:GA_Tk}.
%Note that with a good preconditioner the number $k$ of Lanczos iterations will not be high. Therefore, 
%the increase in computational cost of the modified estimator, caused by eigendecomposition of a larger tridiagonal matrix,
%will be negligibly low.

%-------------
%The extended matrix $\tilde J_{2k-1}$ in~\eqref{eq:GA_Tk}, and therefore the GA Gauss quadrature rule~\eqref{eq:ga_gauss_quad},  
%is fully determined by $k$ steps of the Lanczos procedure. The main difference with the standard Gauss quadrature~\eqref{eq:gauss_quad}
%is that~\eqref{eq:ga_gauss_quad} requires eigendecomposition of a matrix that is almost twice the size of $J_k$.
%However, the corresponding increase in computational cost will be marginal in practice, provided that a good
%preconditioner $T = M^*M$ is chosen, which yields a small number of Lanczos steps. Therefore, 
%since~\eqref{eq:ga_gauss_quad} allows obtaining a higher approximation accuracy 
%%because of an expectedly higher 
%%approximation accuracy 
%at essentially the same cost, 
%we base our final algorithm for estimating the quantity $n\_(A - \tau I)$ on the GA Gauss quadrature
%rule~\eqref{eq:ga_gauss_quad}.    

\section{Preconditioned Arnoldi}\label{sec:arnoldi}

Sometimes, the preconditioner $T$ is not available in a factored form 
$T=MM^\ast$. In this case, it may be necessary to work with 
$T(A - \tau I)$ or $(A - \tau I)T$ directly. 
One possibility is to make use of the fact that $(A-\tau I)T$ is self 
adjoint with respect to an inner product induced by $T$. This property 
allows us to carry out a $T$-inner product Lanczos procedure that produces
\eq{eq:lanT}
(A - \tau I) T X_k = X_{k+1} J_{k+1,k}, \quad X_{k+1}^* T X_{k+1} = I, 
\en
Similarly, we can use a $T^{-1}$-inner product based Lanczos procedure to 
obtain
\eq{eq:lanT1}
T(A - \tau I) Y_k = Y_{k+1} J_{k+1,k}, \quad Y_{k+1}^* T^{-1} Y_{k+1} = I, 
\en
where $Y_k = M^* Q_k$. 
Even though it may appear that we do not need $T$ in a factored form in
either \eqref{eq:lanT} or \eqref{eq:lanT1}, the starting vectors we use
to generate \eqref{eq:lanT} and \eqref{eq:lanT1} are related to $M$. 
In particular, \eqref{eq:lanT} must be generated from $x_1 = M^{-1}q_1$ 
and \eqref{eq:lanT1} must be generated from $y_1 = M^* q_1$, where 
$q_1$ is a random vector with i.i.d entries. 

Another approach is to construct an estimator based on~\eqref{eq:traceT}, where $C = T(A - \tau I)$. 
This will require evaluating the
bilinear form $v^* h(T(A - \tau I))v$, where $h$ is a function of a matrix $T(A-\tau I)$ that has real spectrum 
but is non-Hermitian in standard inner product.    
Similar to the Hermitian case, the matrix functional $v^* h(T(A - \tau I))v$ can be viewed as 
an integral, such that
\eq{eq:int_arnoldi}
v^* h(T(A - \tau I))v = \frac{1}{4\pi^2}\int_{\Gamma} \int_{\Gamma} h(t) v^* (\bar \omega I - (A-\tau I)T)^{-1} (tI - T(A-\tau I))^{-1} v \overline{d \omega} dt,
\en
where $\Gamma$ is a contour that encloses the spectrum of $T(A-\tau I)$ and the bar denotes complex conjugation; see, 
e.g.,~\cite{Hochbruck.Lubich:97}. This integral can be approximated by a quadrature rule based on a few steps 
of the Arnoldi process (Algorithm~\ref{alg:arnoldi}) applied to the preconditioned operator $T(A-\tau I)$ with a starting vector 
$v$~\cite{Calvetti.Kim.Reichel:05, Golub.Meurant:10}. 
%The Arnoldi procedure is outlined in Algorithm~\ref{alg:arnoldi}.
%to mainatian an appropriate distribution of the sampled vectors in the randomized trace estimation.   
\begin{algorithm}[!htbp]
\begin{smaller}
\begin{center}
  \begin{minipage}{5in}
\begin{tabular}{p{0.5in}p{4.5in}}
{\bf Input}:  &  \begin{minipage}[t]{4.0in}
                  Matrix $A - \tau I$, HPD preconditioner $T$, starting vector $v$, and number of steps~$k$.
                  \end{minipage} \\
{\bf Output}:  &  \begin{minipage}[t]{4.0in}
                  Hessenberg matrix $H_{k+1,k}$ and the Arnoldi basis $Q_{k+1} = [q_1, q_2, \ldots, q_{k+1}]$.  
                  \end{minipage}
\end{tabular}
\begin{algorithmic}[1]
\STATE $q_1 \gets v/\|v\|$; $Q_1 \gets q_1$; 
\FOR {$j = 1 \rightarrow k$}
     \STATE $w \gets T(A - \tau I)q_j$;
     \FOR {$i = 1 \rightarrow j$}
        \STATE $h_{i,j} \gets q_i^* w$; $w \gets w - h_{i,j} q_i$; 
     \ENDFOR
      \STATE $h_{j+1,j} \gets \|w\|$; $q_{j+1} \gets w/h_{j+1,j}$; $Q_{j+1} \gets [Q_j, \ q_{j+1}]$;  
\ENDFOR
%\STATE Return $H_{k+1,k}$ in~\eqref{eq:Hk} and $Q_{k+1}$.
\end{algorithmic}
\end{minipage}
\end{center}
\end{smaller}
  \caption{The Arnoldi procedure for $T(A - \tau I)$}
  \label{alg:arnoldi}
\end{algorithm}

Given $q_1 = v/\|v\|$, Algorithm~\ref{alg:arnoldi} produces an 
orthonormal Arnoldi basis $Q_{k+1}$ and an extended upper Hessenberg matrix
\begin{equation}
\label{eq:Hk}
H_{k+1,k}  =  \left [ \begin{array}{cccc}
h_{1,1}  & h_{1,2}  & \ldots & h_{1,k} \\
h_{2,1}   & h_{2,2} & \ddots & h_{2,k} \\
          & \ddots   & \ddots & \vdots  \\
          &                        & h_{k,k-1}   & h_{k,k} \\
          &                        &             & h_{k+1,k} \\
\end{array} \right ] \in \mathbf{R}^{(k+1) \times k},
\end{equation}
such that 
%\eq{eq:arnoldi}
%M(A - \tau I)M^* Q_k = Q_k J_k + \beta_{k+1} q_{k+1} e^*_k, \quad Q_{k+1}^*Q_{k+1} = I,
$T(A-\tau I) Q_k = Q_{k+1} H_{k+1,k}$, $Q_{k+1}^*Q_{k+1} = I$.
%\en
%where $J_{k+1,k}$ is the tridiagonal matrix
%leading block of the matrix
%symmetric tridiagonal matrix
%which is symmetric and tridiagonal.
%; and $e_k \in \mathbf{R}^k$ denotes the $k$-th column of the idenity matrix. 
An Arnoldi quadrature rule for the integral~\eqref{eq:int_arnoldi} is fully determined by 
the $k$-by-$k$ leading submatrix $H_k$ of~\eqref{eq:Hk}. Similar to~\eqref{eq:gauss_quad}, it gives
\eq{eq:arnoldi_quad}
v^* h(T(A - \tau I)) v \approx \|v\|^2 e_1^* h(H_k) e_1 \equiv \sum_{i=1}^{k\_}  w_i t_i, \quad w_i = \|v\|^2 z_i(1), \;  t_i = s_1(i),  
%= \sum_{i=1}^{k} w_i h(\theta_i) \equiv \sum_{i=1}^{k\_} w_i, \quad w_i = \|v\|^2 |z_i(1)|^2,
\en
where $w_i$ are determined by the first components of the (right) eigenvectors $z_1, \ldots, z_{k{\_}}$ of $H_k$ 
associated with its $k\_$ eigenvalues that have negative real parts, 
and $t_i$ is the $i$th entry of the first column of $S = Z^{-1}$.
%\evcomment{discuss $p$-step convergence if $T(A - \tau I)$ has $p$ distinct eigenvalues?}
Similar to Proposition~\ref{prop:mult},
it can be shown that if $T(A - \tau I)$ has $p$ distinct eigenvalues, then \eqref{eq:arnoldi_quad} is 
exact with at most $p$ nodes.   

%\subsection{The algorithm}
Let $H_{k}^{(j)}$ be the upper Hessenberg matrix produced by the Arnoldi process applied to $C = T(A - \tau I)$ with the 
starting vector $v_j$. Then~\eqref{eq:arnoldi_quad} and~\eqref{eq:traceT} yield the 
%preconditioned Arnoldi-type 
estimator
%stochastic estimator $C_{\tau}(k,m)$ of the form 
\eq{eq:rand_prec_quad_arn}
%n_{\_}(A - \tau I) \approx 
A_{\tau}(k,m) = \frac{1}{m} \sum_{j=1}^m \sum_{i=1}^{k_{j}}   w^{(j)}_i t^{(j)}_i, 
\; w^{(j)}_i = \|v_j\|^2z^{(j)}_i(1), \;  t^{(j)}_i = s^{(j)}_1(i), \; v_j \in \mathcal{N}(0,I),  
%\quad E_m \approx n_{\_}(A - \tau I),  
\en
where $z_i^{(j)}(1)$ denotes the first component of the $k_{j}$ unit eigenvectors $z_i^{(j)}$ of $H_{k}^{(j)}$. and $s_i^{(j)}$ is the $i$th entries of the first column of the inverted matrix of eigenvectors of $H_{k}^{(j)}$.  
Similar to~\eqref{eq:rand_prec_quad}, we expect that, for a sufficiently large $m$,  
the real part of $A_\tau(k,m)$ approximates $n_{\_}(A - \tau I)$.
%to the negative inertia of $A - \tau I$. 
The computation of $\mbox{Re}\left( A_\tau(k,m) \right)$ is described in Algorithm~\ref{alg:parnoldi}.
\begin{algorithm}[!htbp]
\begin{smaller}
\begin{center}
  \begin{minipage}{5in}
\begin{tabular}{p{0.5in}p{4.5in}}
{\bf Input}:  &  \begin{minipage}[t]{4.0in}
                  Matrix $A$, shift $\tau$, HPD preconditioner $T$ for $A - \tau I$, number of  steps $k$, and  parameter $m$.
                  \end{minipage} \\
{\bf Output}:  &  \begin{minipage}[t]{4.0in}
                  approximate number $A_{\tau}$ of eigenvalues of $A$ that are less than $\tau$;
                  \end{minipage}
\end{tabular}
\begin{algorithmic}[1]
\STATE $A_{\tau} \gets 0$.
\FOR {$j = 1 \rightarrow m$}
   \STATE Generate $v \sim \mathcal{N}(0,I)$.
   \STATE Run $k$ steps of Arnoldi process in Algorithm~\ref{alg:arnoldi} with 
   the starting vector $v$ to obtain upper Hessenberg matrix $H_k$.
%   \STATE Use $J_{k+1,k}$ to construct $\tilde J_{2k-1}$ in~\eqref{eq:GA_Tk}.
   \STATE Find the eigendecomposition $(\Theta,Z)$ of $H_{k}$.  
   Let $z_1, \ldots, z_{k\_}$ be  unit eigenvectors associated with  negative eigenvalues.
   \STATE Compute $S  = Z^{-1}$. Set $s \gets S(1,:)$. Set $A_{\tau} \gets A_{\tau} + \|v\|^2 \mbox{Re} \left( \sum_{i = 1}^{k\_} w_i s_i \right)$, where $w_i = z_i(1)$, $s_i = s(i)$.
\ENDFOR
\STATE Return $A_{\tau} \gets \left[ A_{\tau}/m \right ]$.
\end{algorithmic}
\end{minipage}
\end{center}
\end{smaller}
  \caption{The preconditioned Arnoldi-type estimator for $n_{\_}(A - \tau I)$}
  \label{alg:parnoldi}
\end{algorithm}

%In order to estimate the eigenvalue count within the interval $[\xi,\eta]$, Algorithm~\ref{alg:parnoldi} must be invoked 
%twice with $\tau = \xi$ and $\tau = \eta$, where each run can be supplied with either same or separate preconditioning procedure. 
The cost of Algorithm~\ref{alg:parnoldi} is comparable to that of Algorithm~\ref{alg:planczos}, and is slightly higher
mainly due to the need to invert the eigenvector matrix of $H_k$. 
%However, if $k$ is small, which is the case when the 
%preconditioner $T$ is sufficiently strong, the difference is negligible. 
%
In contrast to Algorithm~\ref{alg:planczos}, the above described scheme assumes complex arithmetic, because the upper Hessenberg matrix $H_k$ is non-Hermitian 
and can have complex eigenpairs. 
However, for good choices of $T$, the imaginary parts tend to be small in practice as, for a sufficiently large $k$, 
the eigenpairs of $H_k$ converge rapidly to those of $T(A-\tau I)$, which are real. Finally, note that the derivation of the 
estimator~\eqref{eq:rand_prec_quad_arn} assumes an extension of the definition of the step function~\eqref{eq:h}, such that 
$h(x)$ has the value of one on the left half of the complex plane, and is zero elsewhere. 

\section{Polynomial viewpoint}\label{sec:poly}

Let $C = M(A - \tau I)M^*$ or $C = T(A - \tau I)$. Then, we can replace $h(C)$ in~\eqref{eq:traceT} by a 
polynomial approximation $p_{l} (C)$ of degree $l$. There are several ways to choose this polynomial. One
option is to take $p_{l} (C)$ as formal truncated expansion of $h(x)$ in the basis of Chebyshev polynomials.
%truncated to have degree $l$. 
This choice is related the approach described in~\cite{DiNapoli.Polizzi.Saad:2015-TR}.  

The quality of a polynomial approximation $p_l(x)$ of $h(x)$ can be measured
by the difference between $p_l(x)$ and $h(x)$ on the set of eigenvalues of $C$. When the spectrum of $C$ has an arbitrary distribution,
constructing a polynomial that provides the best least squares fit on 
the entire interval containing all eigenvalues, 
as is done in~\cite{DiNapoli.Polizzi.Saad:2015-TR}, is well justified. 

When a good preconditioner is used, the spectrum of 
$C$ tends to cluster around several points on the real line.
Thus, a natural approach would be to choose $p_l$ such that it is only 
close to $h$ in regions that contain  eigenvalue clusters.  It can be 
quite different from $h$ elsewhere. An example of such an approach is 
an interpolating polynomial, e.g.,~\cite{Powell:81},
that interpolates $h$ at eigenvalue clusters. A practical construction of 
such a polynomial is given by the following theorem, which relates the
the interpolation procedure to the Lanczos or Arnoldi process. 

%There are two practical questions related to the construction of such an interpolating polynomial.
%%, however, raises two practical questions. 
%First, it is not immediately clear how to choose the nodes as the spectrum of $M(A - \tau I)M^*$ is not known. 
%The second question concerns the interpolation scheme itself as the traditional approaches, based on the divided
%differences or Vandermonde systems~\cite{Powell:81}, may be numerically unstable.

%The answers are given by the 
%%Our construction is based on the 
%following theorem due to Saad (originally proved for the matrix exponential in~\cite{Saad:92}), which relates the interpolation
%procedure to a Krylov subspace (Lanczos) method for matrix function approximation.

\begin{theorem}[see~\cite{Saad:92, Higham:08}]\label{thm:saad}
Let $Q_{k}$, $T_{k}$ be the orthonormal basis and the projection of
the matrix $C$ generated from  a $k$-step Lanczos (Arnoldi) process,
with the starting vector $v$. 
Then
\begin{equation}\label{eq:lan_interp}
\|v\|Q_{k} f(T_{k}) e_1 = p_{k-1,v}(C) v,
\end{equation} 
where $p_{k-1,v}$ is the unique polynomial of degree at most $k-1$ that interpolates $f$ in the Hermite sense on the spectrum of $T_{k}$. 
\end{theorem} 

The subscript ``$v$'' in $p_{k-1,v}$ is used to emphasize the dependence 
of the polynomial on the staring vector $v$.
%Indeed, different vectors $v$ yield different sets of Ritz values (eigenvalues of $T_{k}$), which lead to different interpolating polynomials.
Note that $T_k$ is a symmetric tridiagonal matrix if $C$ is Hermitian. It 
is upper Hessenberg otherwise.  

Using formula~\eqref{eq:lan_interp}, it is easy to verify that if $C = M(A - \tau I)M^*$, then the bilinear form 
$v^* p_{k-1,v}(C)v$ is exactly the same as the Gauss quadrature rule on the right-hand side of~\eqref{eq:gauss_quad}. 
%\CY{not sure what you mean here, 
%why does Gauss quadrature have to do with any particular form of $C$?}
%\evcomment{This is the matter of whether $C$ is Hermitian or not. For Hermitian, from (23) we obtain Gauss quadrature (12). For non-Hermitian,
%obtain (21).}
Similarly, if $C = T(A - \tau I)$, then $v^* p_{k-1,v}(C)v$ is given by 
the Arnoldi quadrature on the right-hand side of~\eqref{eq:arnoldi_quad}. Hence, both estimators~\eqref{eq:rand_prec_quad}
and \eqref{eq:rand_prec_quad_arn} can be viewed as a stochastic approximation of $\trace \left\{p_{k-1,v} (C)\right\}$,
%for $C = M(A - \tau I)M^*$ or $C = T(A - \tau I)$, respectively,
where $p_{k-1,v}(x)$ is an interpolating polynomial of degree $k-1$ for the step function $h$.

\section{Preconditioning}\label{sec:prec}

The iterative scheme we presented earlier rely on the assumption that 
the operator $T$ is HPD, as this property 
guarantees that the inertia of the original matrix $A - \tau I$ is preserved after preconditioning. Furthermore,
a good choice of $T$ should cluster spectrum of the preconditioned matrix $C$ around
several points in the real axis. 

An ideal HPD preconditioner will result in the preconditioned matrix with only two distinct eigenvalues.  In this case,
by Proposition~\ref{prop:mult}, the Lanczos procedure should terminate
in two steps. 
An example of such an ideal preconditioner is the matrix 
$T = |A - \tau I|^{-1}$, where the absolute value is understood in the 
matrix function sense.  

Clearly, the choice $T = |A-\tau I|^{-1}$ is prohibitively costly in practice. However, it is possible to construct HPD preconditioners
that only approximate $|A - \tau I|^{-1}$. Such a preconditioning strategy was proposed in~\cite{Ve.Kn:13} and is referred to
as the absolute value (AV) preconditioning. 
It was shown in~\cite{Ve.Kn:13} that, e.g., for discrete Laplacian operators, 
AV preconditioners can be efficiently constructed using multigrid (MG). 
%In our numerical experiments, we will use the AV 
%preconditioner of~\cite{Ve.Kn:13} within the proposed algorithms for estimating the eigenvalue counts of a Laplacian.       

Another possible option is to employ the incomplete $LDL^T$ (ILDL) 
factorization. 
%Such a factorization can be constructed, e.g., using the \texttt{sym-ildl} package~\cite{sym-ildl},
%which we will use in our numerical demonstrations. 
Given a matrix $A - \tau I$ and a drop 
tolerance $t$, an ILDL($t$) preconditioner is of the form $T = L^{-*} D^{-1}L^{-1}$, where $L$ is  lower triangular and $D$ is 
block-diagonal with diagonal blocks of size 1 and 2, such that $T \approx (A - \tau I)^{-1}$. 

Clearly, since $A - \tau I$ is indefinite, the ILDL($t$) procedure will generally result in an indefinite 
$T$, which cannot be applied within the preconditioned estimators of this paper. Therefore, we suggest to modify it by 
taking the absolute value of diagonal blocks of $D$, so that $T = L^{-*} |D|^{-1}L^{-1}$. Such a preconditioner is HPD, and the
cost of the proposed modification is marginal. This idea has been motivated by~\cite{Gill.Murray.Ponceleon.Saunders:92}, where a similar approach was used
in the context of full (complete) $LDL^T$ factorization.
% of an approximate matrix.           

Finally, in certain applications, HPD operators are readily available and traditionally used for preconditioning 
indefinite matrices.
For example, this is the case in Density Functional Theory (DFT) based 
electronic structure calculations in which the solutions are expressed
in terms of a linear combination of planewaves. A widely used preconditioner,
often referred to as the Teter preconditioner~\cite{Teter.Payne.Allan:89},
is diagonal in the planewave basis.

\section{Numerical experiments}\label{sec:num}

%Effects of preconditioning, spectrum of preconditoned matrix + (?) Ritz values, accuracy of counts, mesh independence, comparison with LDLT(?).
%Discuss issues with ``off-diagonal energy'' of $T(A-\tau I)$. 
% Show improvement from higher 
%accuracy scheme. 
%This section presents 
We now study the  numerical behavior of the proposed methods for three test problems listed in Table~\ref{tab:test}. The matrix ``Laplace'' represents a standard five-point finite differences (FD) discretization of the 2D 
Laplacian on a unit square with mesh size $h = 2^{-7}$. The problems ``Benzene'' and ``H2'' originates from the DFT based electronic 
structure calculations. The former is a FD discretization of a Hamiltonian operator associated with a ground state benzene molecule\footnote{Available 
in the PARSEC group of the University of Florida Sparse Matrix Collection at 
https://www.cise.ufl.edu/research/sparse/matrices/}, whereas the latter corresponds to a Hamiltonian associated with the hydrogen molecule generated by the 
KSSOLV package~\cite{kssolv:09}. Throughout, our goal is to estimate the quantity $n\_(A - \tau I)$
for a given value of the shift $\tau$. 

\begin{table}[htb]
\centering
\begin{tabular}{cccc||ccc}
Problem & $n$ & $\tau$ & $n\_(A - \tau I)$ & Preconditioner & Estimated $n\_(A - \tau I)$ & $k$\tabularnewline
\hline 
\hline 
Laplace & 16,129 & 3,000 & 226  & no prec. & 232 & 134\tabularnewline
 &  &  &  & ILDL(1e-3) & 216 & 34\tabularnewline
 &  &  &  & ILDL(1e-5) & 229 & 6\tabularnewline
\hline 
Benzene & 8,219 & 5 & 344 & no prec. & 338 & 85\tabularnewline
 &  &  &  & ILDL(1e-5) & 350 & 18\tabularnewline
 &  &  &  & ILDL(1e-6) & 341 & 2\tabularnewline
\hline 
H2 & 11,019 & 0.5 & 19 & no prec & 20 & 50\tabularnewline
 &  &  &  & Teter & 20 & 11\tabularnewline
\hline 
\end{tabular}
\caption{{\small Estimates of $n\_(A - \tau I)$ produced by Algorithm~\ref{alg:planczos} with different preconditioners
and the corresponding numbers $k$ of Lanczos iterations for three test problems.}}
\label{tab:test}
\end{table}

Table~\ref{tab:test} presents the results of applying the Lanczos-type estimator given in Algorithm~\ref{alg:planczos} 
%\CY{Algorithm 1 doesn't give any estimation} 
to the test problems with different preconditioner choices. For the ``Laplace'' and ``Benzene'' matrices, we use the positive 
definite ILDL$(t)$ based preconditioning with different drop tolerance $t$, discussed in the previous section. 
The ILDL factorizations of $A - \tau I$ are obtained using the \texttt{sym-ildl} package~\cite{sym-ildl}.
In the ``H2'' test,
we employ the diagonal Teter preconditioner available in KSSOLV. 
In both cases, the preconditioner is accessible in the factorized form $T = M^*M$.      
The number of random samples $m$ is set to $50$ in all tests.  

In the table, we report estimates of $n\_(A - \tau I)$  produced by Algorithm~\ref{alg:planczos} along with the corresponding 
numbers of Lanczos iterations ($k$) performed at each sampling step.
% to evaluate the bilinear form $v^* h(M(A - \tau I)M^*) v$. 
The reported values of $k$ correspond to the smallest numbers of Lanczos iterations that result in a sufficiently accurate estimate.
The error associated with these approximations have been observed to 
be within $5\%$.

Table~\ref{tab:test} demonstrates that
the use of preconditioning significantly reduces the number of Lanczos 
iterations.  Furthermore, $k$ becomes smaller as the quality of the 
preconditioner, which is controlled by the drop tolerance $t$ in the ILDL$(t)$ 
based preconditioners, improves for the ``Laplace'' and ``Benzene'' tests. 

\begin{figure}[ht]
\begin{center}%
    \includegraphics[width=5.5cm]{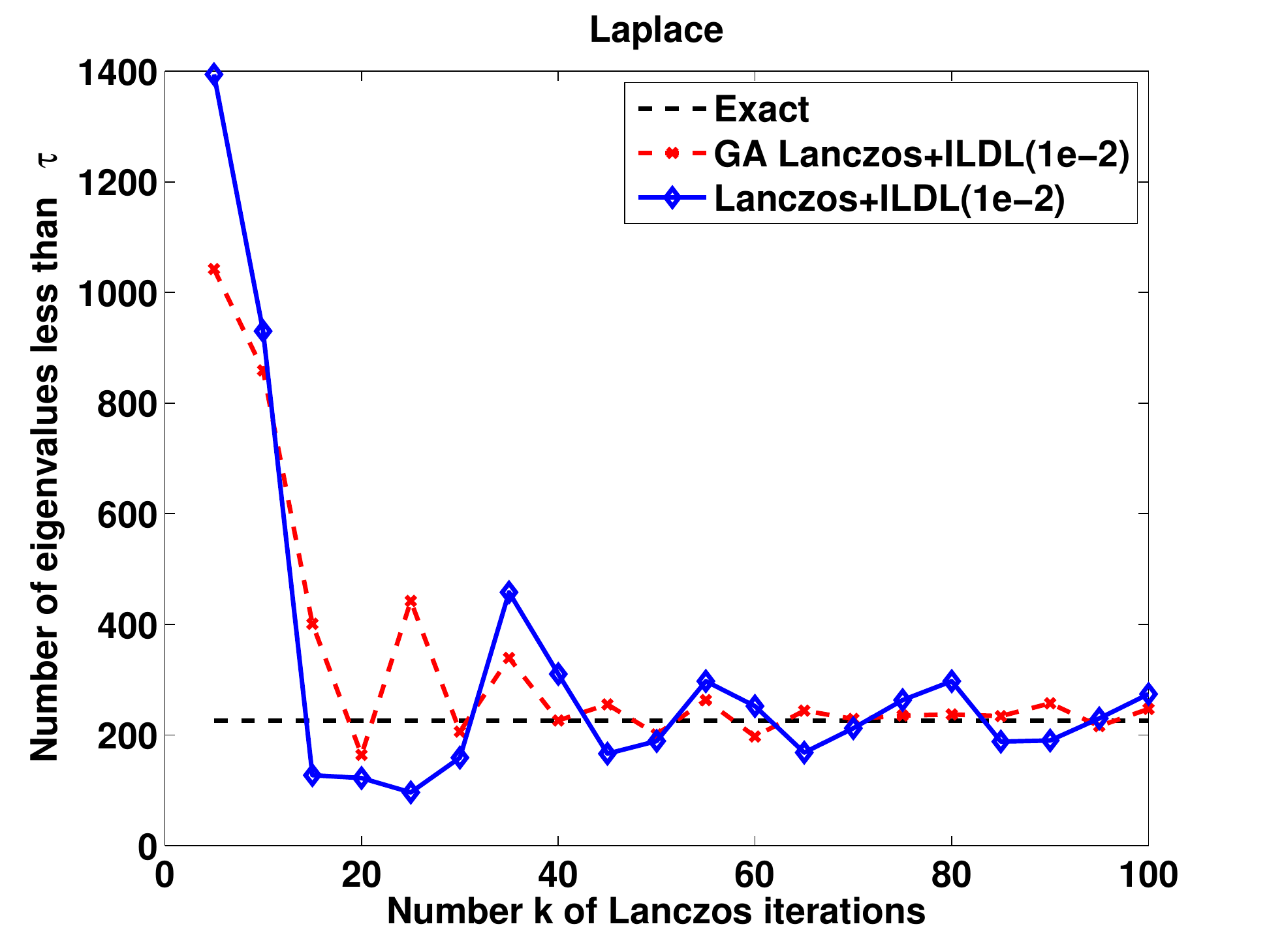}
    \includegraphics[width=5.5cm]{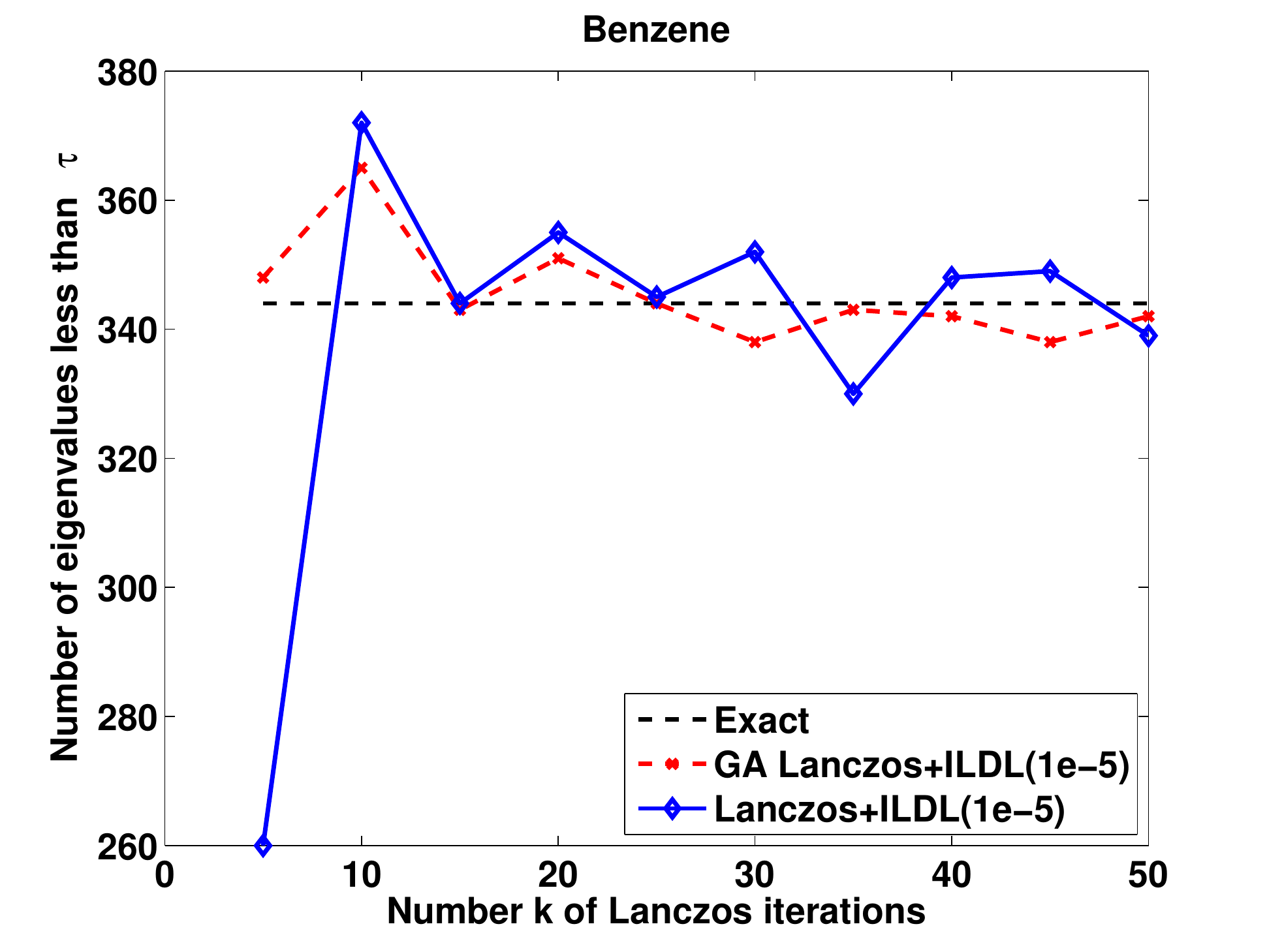}
\end{center}
\caption{
{\small 
Effects of the GA Gauss quadrature of Section~\ref{sec:ga} on the accuracy of the estimator.
}
}
\label{fig:ga}
\end{figure}

Figure~\ref{fig:ga} shows that the quality of the estimates can be further improved by using the GA Gauss quadrature rules
discussed in Section~\ref{sec:ga}. In both plots, the horizontal axis corresponds to the number of Lanczos iterations ($k$) per sampling 
step, and the vertical axis is the corresponding estimate of $n\_(A - \tau I)$. It can be seen that 
the estimator based on the GA Gauss quadrature (referred to as ``GA Lanczos'') is generally more accurate for the two test problems, 
with the accuracy difference being especially evident for smaller values of $k$. 

In the context of linear systems arising from discretizations of partial differential equations, 
an important property of preconditioning is that it allows maintaining the same number of iterations needed to obtain solution
regardless of problem size.
% and conditioning. 
A similar phenomenon can be observed when estimating $n\_(A - \tau I)$ using the 
preconditioned methods of this paper.      

\begin{table}[htb]
\centering
\begin{tabular}{lr}

\begin{tabular}{|c||c|c|c|c|c|}
\hline 
$h$ & $2^{-6}$ & $2^{-7}$ & $2^{-8}$ & $2^{-9}$ & $2^{-10}$\tabularnewline
\hline 
\hline 
Chebyshev & 8 & 14 & 34 & 62 & 80\tabularnewline
\hline 
Arnoldi+AV & 16 & 16 & 18 & 19 & 16\tabularnewline
\hline 
\end{tabular}&
\hspace{1cm}
\begin{tabular}{|c||c|c|c|c|c|}
\hline 
ecut (Ry) & 25 & 50 & 75 & 100 & 125\tabularnewline
\hline 
\hline 
Chebyshev & 52 & 78 & 76 & 99 & 124\tabularnewline
\hline 
Lanczos+Teter & 8 & 8 & 11 & 8 & 8\tabularnewline
\hline 
\end{tabular}

\end{tabular}
\caption{{\small Independence of preconditioned Arnoldi- and Lanczos-type estimators for $n\_(A - \tau I)$ on the discretization parameter for the ``Laplace'' (left)
and ``H2'' (right) problems.}}
\label{tab:meshind}
\end{table}

In Table~\ref{tab:meshind} (left) we consider a family of discrete Laplacians, whose size and condition numbers increase as the 
mesh parameter $h$ is refined. For each of the matrices, we apply the Arnoldi-type estimator of Algorithm~\ref{alg:parnoldi} with the MG AV preconditioner 
from~\cite{Ve.Kn:13} and, similar to above, report the smallest numbers $k$ of Arnoldi iterations per sampling step needed to obtain a sufficiently
accurate estimate (within $5\%$ error) of $n\_(A - \tau I)$. The results are compared against those of an unpreconditioned estimator based 
on~\eqref{eq:traceT}, where $C = A$ and the step function $h(A)$ is replaced by its least-squares polynomial approximation of degree $k$ constructed 
using the basis of Chebyshev polynomials. The latter (referred to as ``Chebyshev'') is essentially the approach proposed in~\cite{DiNapoli.Polizzi.Saad:2015-TR}.  

It can be seen from the table, that Algorithm~\ref{alg:parnoldi} with the AV preconditioner exhibits behavior that is independent of $h$.
Regardless of the problem size and conditioning, the number of Arnoldi steps stays (roughly) the same (between 16 and~19).  

In Table~\ref{tab:meshind} (right) we report a similar test for a sequence of ``H2'' problems obtained by increasing the kinetic energy cutoff 
(ecut) from 25 to 125 Ry in the plane wave discretization. This gives Hamiltonian matrices with sizes ranging from $1,024$ to $23,583$.
Again, we observe that the behavior of the Lanczos-type estimator in Algorithm~\ref{alg:planczos} with the Teter
preconditioner~\cite{Teter.Payne.Allan:89} is essentially independent of the discretization parameter, whereas the ``Chebyshev'' approach tends to 
require higher polynomial degrees as the problem size grows. 
\vspace*{0.2in}

\noindent {\bf Acknowledgement}
Support for this work was provided through Scientific Discovery through Advanced Computing (SciDAC) program funded by U.S. Department of Energy, Office of Science, Advanced Scientific Computing Research.

\bibliographystyle{plain}
\bibliography{local}

\end{document}